%% file: qcdim.tex
\documentclass[12pt]{article}
\usepackage{geometry}                
\usepackage{graphicx,color}
\usepackage{amssymb,amsmath,amsthm,mathrsfs}
\usepackage[all,cmtip]{xy}
\usepackage{epstopdf, comment}
\usepackage{esint}

\usepackage[pdftex,bookmarks,pdfnewwindow,plainpages=false,unicode,pdfencoding=auto]{hyperref}

 \hypersetup{
pdfauthor={Oleg Ivrii},
pdftitle={Quasicircles of dimension 1+k\texttwosuperior\ do not exist},
pdfsubject={Quasiconformal geometry},
bookmarksdepth={4}
}

\DeclareMathOperator{\Hdim}{H.dim }
\DeclareMathOperator{\supp}{supp }

\newtheorem{lemma}{Lemma}[section]
\newtheorem{theorem}{Theorem}[section]

\theoremstyle{remark}
\newtheorem*{remark}{Remark}

\numberwithin{equation}{section}

\DeclareMathOperator{\im}{Im}
\DeclareMathOperator{\re}{Re}
\DeclareMathOperator{\cl}{clos}
\DeclareMathOperator{\inte}{int}

\DeclareMathOperator{\aut}{Aut}

\DeclareMathOperator{\osc}{osc}

\DeclareMathOperator{\bSigma}{\mathbf \Sigma}
\DeclareMathOperator{\bH}{\mathbf H}
\DeclareMathOperator{\Mdim}{M. dim}

\DeclareMathOperator{\hyp}{hyp}

\DeclareMathOperator{\per}{per}
\DeclareMathOperator{\ei}{I}

\newcommand{\dzy}{\,\frac{|dz|^2}{y}}
\newcommand{\rhoH}{\rho_{\mathbb{H}}}
\newcommand{\Hbar}{\overline{\mathbb H}}

\DeclareMathOperator{\id}{Id}
\DeclareMathOperator{\bel}{dil. }

\title{Quasicircles of dimension $1+k^2$ do not exist}
\author{Oleg Ivrii}
\date{April 15, 2016}

\begin{document}

\maketitle
\begin{abstract}
A well-known theorem of S.~Smirnov states that the Hausdorff dimension of a  $k$-quasicircle is at most $1+k^2$. Here, we show that
 the precise upper bound  $D(k) = 1+\Sigma^2 k^2 + \mathcal O(k^{8/3-\varepsilon})$
 where  $\Sigma^2$ is the maximal asymptotic variance of the Beurling transform, taken over the unit ball of $L^\infty$. 
The quantity $\Sigma^2$ was introduced  in a joint work with K.~Astala, A.~Per\"al\"a and I.~Prause where it was proved that
 $0.879 < \Sigma^2 \le 1$, while recently, H.~Hedenmalm discovered that surprisingly $\Sigma^2 <1$.
We deduce
the asymptotic expansion of $D(k)$ from a more general statement relating the universal bounds for the integral means spectrum and the asymptotic variance of conformal maps. 
Our proof combines fractal approximation techniques with
 the classical argument of J.~Becker and Ch.~Pommerenke  for estimating integral means.
\end{abstract}

\section{Introduction}

 Let $D(k)$ denote the maximal Hausdorff dimension of a $k$-quasicircle, the image of the unit circle under a $k$-quasiconformal mapping of the plane.
The first non-trivial bound (with the right growth rate) was given in 1987 by Becker and Pommerenke \cite{BP} who proved that
$1 + 0.36 \, k^2 \le D(k) \le 1 + 37\, k^2$ if $k$ is small. In 1994, in his landmark work \cite{astala-area} on the area distortion of quasiconformal mappings, K.~Astala suggested that the correct bound was
\begin{equation}
\label{eq:smirnov-bound}
D(k) \le 1+k^2, \qquad 0 \le k <1.
\end{equation}
Using a clever variation of Astala's argument, S.~Smirnov \cite{smirnov} showed that  the bound 
(\ref{eq:smirnov-bound}) indeed holds. A systematic investigation of the sharpness of (\ref{eq:smirnov-bound}) was
initiated in \cite{AIPP} where the quantity
\begin{equation}
\Sigma^2 \, := \, \sup_{|\mu| \le \chi_\mathbb{D}} \sigma^2(\mathcal S\mu)
\end{equation}
was introduced. Here,
\begin{equation} 
\label{eq:beurling}
 \mathcal S \mu(z) \, =\, -\frac{1}{\pi} \int_{\mathbb{D}} \frac{\mu(\zeta)}{(\zeta-z)^2} \, |d\zeta|^2
 \end{equation} 
 denotes the {\em Beurling transform} of $\mu$ and
 \begin{equation}
 \label{eq:sigma2-def}
 \sigma^2(g) = \lim_{R \to 1^+} \frac{1}{2\pi|\log(R-1)|} \int_{|z|=R} |g(z)|^2 \, |dz|
 \end{equation} 
 is the {\em asymptotic variance} of a Bloch function $g \in \mathcal B(\mathbb{D}^*)$ on the exterior unit disk.
 The motivation for studying $\Sigma^2$ comes from the work of McMullen \cite{mcmullen} who showed
 that in dynamical cases (i.e.~when $\mu$ is invariant under a co-compact Fuchsian group or eventually-invariant under a Blaschke product), one has 
 an explicit connection between Hausdorff dimension and asymptotic variance:
 \begin{equation}
\label{eq:mcmullen-local}
2 \frac{d^2}{dt^2} \biggl |_{t=0} \Hdim w^{t \mu}(\mathbb{S}^1) = \sigma^2(\mathcal S\mu).
\end{equation}
 
In \cite{AIPP}, it was established that $\liminf_{k \to 0} (D(k) - 1)/k^2 \ge \Sigma^2$ and
$
0.879 \le \Sigma^2 \le 1,
$
while recently, H.~Hedenmalm surprisingly proved that $\Sigma^2 < 1$, see \cite{hedenmalm}.
 Here, we complete this ``trilogy'' by showing: 
 \begin{theorem}
 \label{dim-exp}
 $$
 D(k) = 1+\Sigma^2 k^2 + \mathcal O(k^{8/3-\varepsilon}), \qquad \text{for any }\varepsilon > 0.
 $$
 \end{theorem}
 In particular, Smirnov's bound is not sharp for small $k$.
  Istv\'an Prause informed me (private communication) that 
one can use the methods of \cite{prause-smirnov, smirnov} to show that 
 this implies that $D(k) < 1 + k^2$ for all $0 < k < 1$.

\subsection{Integral means spectra}

The aim of {\em geometric function theory} is to understand the geometric complexity of the
boundary of a  simply-connected domain $\Omega \subset \mathbb{C}$ in terms of the analytic complexity
of the Riemann map $f: \mathbb{D} \to \Omega$.
For domains with rough boundaries, the relationship between $f$ and $\partial \Omega$ may be quantified using several
geometric characteristics. 
 One notable characteristic is the {\em integral means spectrum}
$$
\beta_f(p) = \limsup_{r \to 1} \frac{\log \int_{|z|=r}|f'(z)^p| \, d\theta}{\log \frac{1}{1-r}}, \qquad p \in \mathbb{C}.
$$
The importance of the spectrum $\beta_f(p)$ lies in the fact that it is  Legendre-dual to the multifractal spectrum of harmonic measure \cite{makarov99, binder}. 
 Taking the supremum of $\beta_f(p)$ over bounded simply-connected domains, one obtains the {\em universal integral means spectrum}
 $$B(p) = \sup \beta_f(p).$$ 
Apart from various estimates \cite{hs, jones-survey},
 not much is rigorously known about the qualitative features of $B(p)$.
For instance, it is expected that $B(p) = B(-p)$ is an even function. However, while $B(2)=1$ is an easy consequence of the area theorem, the statement ``$B(-2) =1$'' is equivalent to the Brennan conjecture, which is a well-known and difficult open problem.
Nor is it known whether $B(p) \in C^1$, let alone real-analytic.
In this work, we are concerned with the quadratic behaviour of $B(p)$ near the origin.

It will be convenient for us to work with conformal maps defined on  the exterior unit disk $\mathbb{D}^* = \{z: |z| > 1\}$. Unless stated otherwise, we
 assume that 
conformal maps are in {\em principal normalization}\,, satisfying $\varphi(z) = z + \mathcal O(1/|z|)$ near infinity.

Let $\bSigma_k$ be the collection of conformal maps that admit 
$k$-quasiconformal extensions to the complex plane with dilatation at most $k$.
Maximizing over $\bSigma_k$, we obtain the spectra $B_k(p) := \sup_{\varphi \in \bSigma_k} \beta_\varphi(p)$.
We show:

\begin{theorem}
\label{main-thm}
\begin{equation}
\label{eq:main-eq}
\lim_{p \to 0}
  \frac{B_k(p)}{|p|^2/4} \, = \,  \Sigma^2(k)  \, := \, \sup_{\varphi \in \bSigma_k} \sigma^2(\log \varphi').
\end{equation}
\end{theorem}

\begin{theorem}
\label{main-thm2}
If $k \to 0$ and $k|p| \to 0$, then
$$
\lim \frac{B_k(p)}{k^2 |p|^2 /4} \, = \, \Sigma^2 \, := \, \sup_{|\mu| \le \chi_\mathbb{D}} \sigma^2(\mathcal S\mu).
$$
 \end{theorem}
 
\begin{theorem}
\label{sigmak-convex}

{\em (i)} $\Sigma^2(k)/k^2$ is a non-decreasing convex function on $[0,1)$. 

{\em (ii)} Furthermore, $\Sigma^2(k) = \Sigma^2 k^2 + \mathcal O(k^3)$ as $k \to 0$.
\end{theorem}

Together with Hedenmalm's estimate, Theorem \ref{main-thm2} contradicts the very general conjecture

\centerline{``$B_k(p) = k^2 p^2/4$ for
all $k \in [0,1)$ and $p \in [-2/k,2/k]$''}
\smallskip

\noindent from \cite{jones-survey, prause-smirnov}.
However, since we do not know whether or not \,$\lim_{k \to 1^-} \Sigma^2(k) \stackrel{?}{=}  1$, we cannot rule out Kraetzer's conjecture which asserts only that
\smallskip

\centerline{``$B(p) = p^2/4$ for $p \in [-2,2].$''}
\smallskip

\noindent It is currently known that $0.93 < \lim_{k \to 1^-} \Sigma^2(k) < (1.24)^2$. We refer the reader to \cite[Section 8]{AIPP} for the lower bound
and to \cite{hedenmalm-shimorin, HK} for the upper bound.

\medskip

From the point of view of universal Teichm\"uller space, 
 $\mathcal S \mu$ is the infinitesimal analogue
of $\log \varphi'$. Indeed, if $w^{t\mu} \in \bSigma_k$ is the principal solution to the Beltrami equation $\overline{\partial} w = t \mu \, \partial w$,
then $\log (w^{t\mu})' \approx t \mathcal S \mu$. Further remarks will be given in Section \ref{sec:background}.
McMullen's identity (\ref{eq:mcmullen-local})
 admits a global form, described in \cite{AIPP}:
\begin{theorem}
\label{mcmullen-global}
 If
 $\partial \Omega$ is a Jordan curve, invariant under a hyperbolic dynamical system, e.g.~a Julia set or a limit set of a quasi-Fuchsian group, then
$\beta_\varphi$ is real-analytic and
\begin{equation}
\label{eq:mcmullen-global}
(1/2) \cdot \beta_\varphi''(0) = \sigma^2(\log \varphi').
\end{equation}
\end{theorem}

\begin{remark}
For general domains, (\ref{eq:mcmullen-global}) need not hold: for instance, one can take a fractal domain 
which satisfies McMullen's identity and replace an
arc $\gamma \subset \partial \Omega$ by a smooth curve. Then $\beta_\varphi''(0)$ does not change, but the asymptotic variance goes down by a definite factor,
depending on the harmonic measure of $\gamma$. For more counterexamples, see the works \cite{BaMo, le-zinsmeister}.
\end{remark}

\subsection{From integral means to dimensions of quasicircles}

The implication (Theorem \ref{main-thm2} $\Rightarrow$ Theorem \ref{dim-exp}) follows from the relation 
\begin{equation}
\label{eq:beta-dimension}
\beta_{\varphi}(p) = p-1 \quad \Longleftrightarrow \quad p = \Mdim \varphi(\mathbb{S}^1), \qquad \varphi \in \bSigma_k,
\end{equation}
see \cite[Corollary 10.18]{Pomm}.
Here, two facts are tacitly being used: first, the work of Astala \cite{Arevista} shows that $D(k)$ may be characterized with Minkowski dimension
in place of Hausdorff dimension. One may view Astala's result as a {\em fractal approximation} theorem: when evaluating $D(k)$, it suffices
to take the supremum over certain  Ahlfors regular $k$-quasicircles for which the Hausdorff and Minkowski dimensions coincide.
This theme will recur throughout this work. 

Secondly, one can take a quasiconformal map that is conformal to one side and anti-symmetrize it in the spirit
of \cite{kuhnau, smirnov} to reduce its dilatation. 
More precisely, a Jordan curve $\gamma$ is a $k'$-quasicircle if and only if
it can be represented as $\gamma = m( \varphi(\mathbb{S}^1))$ where $m$ is a M\"obius transformation and $\varphi \in \bSigma_k$ with
\begin{equation}
\label{eq:anti-symmetrization}
k = \frac{2k'}{1 +(k')^2}.
\end{equation}
This accounts for the discrepancy in the factor of $4$ in Theorems  \ref{main-thm2} and \ref{dim-exp}.

\subsection{A sketch of proofs}

The proofs of Theorems  \ref{main-thm} and \ref{main-thm2} follow the argument of Becker and Pommerenke \cite{BP}, 
except we
use an $L^2$ bound for the non-linearity $n_\varphi = \varphi''/\varphi'$ instead of the $L^\infty$ bound
\begin{equation}
\biggl |\frac{2n_\varphi}{\rho_*}(z) \biggr| \le 6 k, \qquad \varphi \in \bSigma_k.
\end{equation}
 Here, $\rho_*(z) = 2/(|z|^2-1)$ is the density of the hyperbolic metric on the exterior unit disk. 
 By a {\em box} in $\mathbb{D}^*$, we mean an annular rectangle of the form
$$
B = \{z : r_1 < |z| < r_2, \ \theta_1 < \arg z < \theta_2 \},
$$
while the notation $\fint_B \dots \, \rho_*(z)|dz|^2$ suggests that we consider the average with respect to the measure $\rho_*(z)|dz|^2$.
 
 \begin{theorem}
\label{thm-mM}
Suppose $\varphi \in \bSigma_k$ is a conformal map which satisfies
 \begin{equation}
 \label{eq:boxcart-global2}
m \, \le \,  \fint_{B} \, \biggl |\frac{2n_\varphi}{\rho_*}(z) \biggr |^2 \, \rho_*(z) |dz|^2 \, \le \, M
 \end{equation}
for all sufficiently large boxes $B$ in the exterior unit disk.  Then,
$$m  \, \le \,  \liminf_{|p| \to 0} \frac{\beta_\varphi(p)}{|p|^2/4}
\, \le \,
\limsup_{|p| \to 0} \frac{\beta_\varphi(p)}{|p|^2/4} \, \le \, M.
$$
\end{theorem}

To summarize, {\em upper bounds} on box averages yield {\em upper bounds} on integral means, while {\em lower bounds} on box averages yield {\em lower bounds} on   integral means.

The sharp upper bound for the average non-linearity in  (\ref{eq:boxcart-global2}) will be computed using the quasiconformal fractal approximation technique of \cite{AIPP}.
We require slightly more general considerations than those given in \cite{AIPP}, so we present these ideas in full detail. These arguments
take up Sections \ref{sec:inf-locality}--\ref{sec:gl-locality}. The term ``fractal approximation'' comes from the original application of these ideas
which was:

\begin{theorem}
\label{AIPP-fat}
Let $M_{\ei}$ be the class of Beltrami coefficients that are {\em eventually-invariant} under $z \to z^d$ for some $d \ge 2$, i.e.~satisfying
$(z^d)^*\mu = \mu$ in some open neighbourhood of the unit circle. Then,
\begin{equation}
\label{eq:AIPP-fat}
\Sigma^2 = \sup_{\mu \in M_{\ei}, \ |\mu| \le \chi_{\mathbb{D}}} \sigma^2(\mathcal S\mu),
\end{equation}
\begin{equation}
\label{eq:AIPP-fat2}
\Sigma^2(k) = \sup_{\mu \in M_{\ei}, \ |\mu| \le k \cdot \chi_{\mathbb{D}}} \sigma^2 \bigl (\log (w^\mu)' \bigr ).
\end{equation}
\end{theorem}

The {\em Box Lemmas} (Lemmas \ref{boxcart} and \ref{boxcart-global}) may be regarded as the quantitative forms of the statements (\ref{eq:AIPP-fat}) and
(\ref{eq:AIPP-fat2}).

\subsection{Applications to dynamical systems}

In dynamical cases, using the ergodicity of the geodesic flow on the unit tangent bundle $T_1 X$ (Fuchsian case) or Riemann surface lamination $\hat X_B$ (Blaschke case), it is not difficult to show that
\begin{equation}
\label{eq:ergodicity}
\sigma^2(\log\varphi')  - \varepsilon \, < \, \fint_{B} \, \biggl |\frac{2n_\varphi}{\rho_*}(z) \biggr |^2 \, \rho_* |dz|^2 \, < \,  \sigma^2(\log\varphi') + \varepsilon
\end{equation}
holds for all sufficiently large boxes $B$.
In view of Theorem \ref{thm-mM},
this gives an elementary proof of McMullen's identity (Theorem \ref{mcmullen-global}) which does not use thermodynamic formalism. 

\begin{remark} 
 To be honest, this argument does not show the true differentiability of the functions $p \to \beta_\varphi(p)$ and $t \to \Hdim w^{t \mu}(\mathbb{S}^1)$, only that
  $\beta_\varphi(p) = \sigma^2(\log \varphi')p^2/4 + o(p^2)$ and
$\Hdim w^{t \mu}(\mathbb{S}^1) = \sigma^2(\mathcal S\mu)t^2/4 + o(t^2)$.
\end{remark}

\subsection{Beltrami coefficients with sparse support}
When studying thin regions of Teichm\"uller space, it is natural to consider Beltrami coefficients that are sparsely supported.  For applications, see \cite{bishop-bigdef, ivrii-phd, McM-cusps}.
Suppose $\mu \in M(\mathbb{D})$ is a Beltrami coefficient supported on a ``garden'' 
$\mathcal G = \bigcup A_j$ where:

\begin{itemize}
\item[$(1)$] Each $A_j$ satisfies the {\em quasigeodesic property} -- i.e.~is located within hyperbolic distance $S$ of a geodesic segment $\gamma_j$.

\item[$(2)$] {\em Separation property.} The hyperbolic distance $d_{\mathbb{D}}(\gamma_i, \gamma_j) > R$ is large.
\end{itemize}
\begin{theorem}
\label{sparse-thm}
If $\mu$ is a Beltrami coefficient with sparse support, then $$\Mdim f^{t \mu}(\mathbb{S}^1) \le 1 + C(S)e^{-R/2} |t|^2,$$  for $|t| < t_0(S,R)$ small.
\end{theorem}

\subsection{Related results and open problems}

After the first version of the paper was written, the author realized that the average non-linearity is related to the local variance
of a dyadic martingale associated to a Bloch function, introduced by Makarov.
These martingale arguments \cite{ivrii-mak} give a quicker route to the main results of this paper as well as give additional characterizations of $\Sigma^2(k)$ in terms of the constant in Makarov's law of iterated logarithm and the transition parameter for the singularity of harmonic measure.
Nevertheless, the Becker-Pommerenke argument is in some ways stronger: martingale techniques give a weaker error term
in the expansion of $D(k)$ and are not applicable (at least in a direct way) to study quasiconformal mappings whose dilation has small support
(Theorem \ref{sparse-thm}).

In a recent work \cite{hedenmalm-atvar}, Hedenmalm studied the notion of ``asymptotic tail variance'' of Bloch functions to show the 
estimate 
\begin{equation}
\label{eq:hed2-thm}
B_k(p) \le (1+7k)^2 \cdot \frac{k^2|p|^2}{4}, \qquad |p| \le \frac{2}{k(1+7k)^2}.
\end{equation}
However, the arguments of this paper are only effective when the product $k|p|$ is small.
It would be natural to interpolate between
Theorem \ref{main-thm} and (\ref{eq:hed2-thm}) in the range $0 \le k|p| \le 2$.

\noindent To conclude the introduction, we mention several open problems:
\begin{enumerate}
\item Is it true that $\lim_{p \to 0} 4 \, B(p)/p^2 = \lim_{k\to 1^-}\Sigma^2(k)$\,?

\item
Are the functions $B_k(p)$ and $D(k)$ differentiable on an interval $(-\epsilon,\epsilon)$?

\item 
 Is it true that $D(k) = \Sigma^2 k^2 + a_3 k^3 + o(k^3)$ for some $a_3 \in \mathbb{R}$\,?

\item For a Bloch function $g \in \mathcal B(\mathbb{D}^*)$, let
$$
\sigma^2(g, R) = \frac{1}{2\pi|\log(R-1)|} \int_{|z|=R} |g(z)|^2 \, |dz|
$$
and $\Sigma^2_R := \sup_{|\mu| \le \chi_{\mathbb{D}}} \sigma^2(g,R)$.
In \cite{AIPP}, it was proved that
$$|\Sigma^2_R - \Sigma^2| \le C \cdot \log \frac{1}{R-1}, \qquad 1 < R < 2.$$ Does one have ``exponential mixing''
$|\Sigma^2_R - \Sigma^2| \le C \cdot  \frac{1}{(R-1)^\gamma}$ for some $\gamma > 0$?
\end{enumerate}

\subsection*{Acknowledgements}

The author wishes to thank K.~Astala, H.~Hedenmalm, I.~Kayumov, A. Per\"al\"a and I.~Prause for stimulating conversations.
The research was supported by the Academy of Finland, project no.~271983.

\input{background}

\input{qcdimbulk}

\section{Estimating integral means}
\label{sec:integral-means}

In this section, we review the Becker-Pommerenke argument which gives the bound $D(k) \le 1 +  36\, k^2 + \mathcal O(k^3)$. We then 
 modify the argument to take advantage of the box lemma, allowing us to replace $36$ with $\Sigma^2$. 

\subsection{Becker-Pommerenke argument}

For convenience, we work in the subclass  $\mathbf{H}^{1}_k \subset \mathbf{H}_k$  of maps that satisfy the periodicity condition $f(z+1) = f(z)+1$.
In view of the exponential transform (\ref{eq:exponential-transform}), we are secretly working with a conformal map of the exterior unit disk.
Define
\begin{equation}
I_p(y) := \int_{0}^{1} |f'(x+iy)^p | \, dx, \qquad p \in \mathbb{C}.
\end{equation}
Since $\frac{\partial}{\partial y} = i(\frac{\partial}{\partial z} - \frac{\partial}{\partial \overline{z}})$, differentiation under the integral sign
shows
\begin{equation}
 I_p'(y) \, = \, -p \int_0^1 |f'(x+iy)^p | \, \im\biggl ( \frac{f''}{f'} \biggr) dx.
\end{equation}
In view of the normalization, $f' \to 1$, $f'' \to 0$ as $y \to \infty$, in which case $ I_p'(y) \to 0$.
To estimate $I_p$, we use a variant of Hardy's identity on the upper half-plane which says that for
a holomorphic function $g(z)$ with $g(z+1) =g(z)$,
\begin{equation}
\frac{d^2}{dy^2} \int_{0}^{1} |g(x+iy)| \, dx =
  \int_{0}^{1} \Delta | g(re^{i\theta}) | \, dx.
\end{equation}
Indeed, $\frac{d^2}{dx^2}  \int_{0}^{1} |g(x+iy)| \, dx = 0$ by periodicity.
Applying Hardy's identity to $f'(z)^p$ gives
\begin{equation}
\label{eq:beck-pomm}
 I_p''(y) \, = \, |p|^2 \int_0^1 |f'(x+iy)^p | \,  \biggl | \frac{f''}{f'} \biggr |^2 dx.
\end{equation}
In particular, $I_p'(y)$ is increasing as $y \to \infty$ which shows that $I_p'(y) \le  0$. 
Replacing non-linearity by its supremum bound, we obtain
\begin{equation}
\label{eq:hardy-bound2}
I_p''(y) \le \frac{9\, k^2|p|^2}{y^2} \, I_p(y).
\end{equation}
From the differential inequality (\ref{eq:hardy-bound2}) together with $I_p(y) \ge 0$, $I'_p(y) \le 0$, it follows that
$$
I_p(y) \le C \cdot y^{-9 k^2|p|^2}, \qquad k \in [0,1), \quad p \in \mathbb{C},
$$
see Lemma \ref{diff-ineq}(i) below. In other words, $\beta_{f}(p) \le 9\, k^2|p|^2$. Anti-symmetrization (\ref{eq:anti-symmetrization}) and the equation $\beta_f(\Mdim \partial \Omega) = \Mdim \partial \Omega - 1$
 yield the dimension bound $D(k) \le 1 +  36\, k^2 + \mathcal O(k^3)$.

\subsection{A differential inequality}

To make use of (\ref{eq:hardy-bound2}), we used an elementary fact about differential inequalities.
 If necessary, the reader may consult \cite[Proposition 8.7]{Pomm}.

\begin{lemma}
\label{diff-ineq}
{\em (i)} Suppose $u(y)$ is a $C^2$ function on $(0,y_0)$ with $$u \ge 0 \qquad \text{and} \qquad  u' \le 0$$
satisfying
\begin{equation}
\label{eq:diff-inequality}
u''(y) \le \frac{\alpha u}{y^2},
\end{equation}
for some constant $\alpha > 0$. Then, $$u(y) \, \le \, v(y) \, = \, C y^{-\beta}, \quad \text{where } \beta^2 + \beta = \alpha,$$
when $C > 0$ is sufficiently large so that $u(y_0) \le v(y_0)$ and $|u'(y_0)| \le |v'(y_0)|$.

\medskip

\noindent
{\em (ii)} Conversely, if (\ref{eq:diff-inequality}) is replaced by
\begin{equation}
\label{eq:diff-inequality2}
u''(y) \ge \frac{\alpha u}{y^2},
\end{equation}
then
$$u(y) \, \ge \, v(y) \, = \, c y^{-\beta},$$
when $c > 0$ is sufficiently small so that $u(y_0) \ge v(y_0)$ and $|u'(y_0)| \ge |v'(y_0)|$.
\end{lemma}

\begin{remark}
When $\alpha > 0 $ is small, 
$
\beta = \alpha - \mathcal O(\alpha^2),
$
in which case $\alpha \approx \beta$.
\end{remark}

\subsection{Averaging over annuli}

Using the box lemma, we can give an improvement in the argument of Becker and Pommerenke.
Given an integer  $n \ge 1$, let $A(y)$ denote the rectangle $[0,1] \times [y/n, y]$ and $R = \log n$ be its hyperbolic height.
Consider the function
 \begin{equation}
 \label{eq:u-def}
u(y) \, := \,  \int_{y/n}^y I_p(h) \, \frac{dh}{h} \, = \, \int_1^{1/n} I_p(yh) \cdot \frac{dh}{h}.
\end{equation}
Since $\| \log f' \|_{\mathcal B(\mathbb{H})} \le 6$, we have $u(y) \asymp I_p(y)$, which allows us to
 compute the integral means spectrum of $f$ by measuring the growth of $u(y)$ as $y \to 0^+$.
 Differentiating (\ref{eq:u-def}) twice gives
\begin{equation}
\label{eq:improved-smirnov}
u''(y) \, = \,
  \int_{y/n}^y I_p''(h) \, \frac{dh}{h} \, = \, 
\frac{|p|^2}{4y^2}  \int_{A(y)}  | f'(z)^p | \,  \biggl | \frac{2n_f}{\rho_{\mathbb{H}}} \biggr |^2 \, \frac{dh}{h} \cdot dx.
\end{equation}

\begin{lemma}
\label{lemma-RHS}
Suppose that the average non-linearity of $f \in \mathbf{H}_k^1$ over any box $\square \in \widehat{\mathscr G_n}$ with
$\square \subset [0,1] \times [0, y_0]$ is bounded above by $M$.
Then,
\begin{equation}
\label{eq:lemma-RHS}
u''(y) \le M \exp(CRk|p|) \cdot \frac{|p|^2}{4y^2} \cdot u(y),
\qquad y \in (0, y_0),
\end{equation}
\end{lemma}

\begin{proof}
We partition $A(y)$ into $\widehat{\mathscr G_n}$-boxes $B_1, B_2, \dots, B_{N(y)}$. (The number of boxes 
is roughly proportional to the hyperbolic length of the segment $\{x+iy : 0 \le x \le 1\}$, but we will not use this.)
Since $f$ has a $k$-quasiconformal extension to the plane, $\| \log f' \|_{\mathcal B} \le 6 k$. As the hyperbolic diameter of a box in $\widehat{\mathscr G_n}$ is comparable to $R$, the multiplicative oscillation of $|f'(z)^p|$ in 
each box $B_j$ is at most
$$
\osc_{B_j} |f'(z)^p| := \sup_{z_1, z_2 \in B_j} \log \frac{ |f'(z_1)^p|}{ |f'(z_2)^p|} \le C_1 R k|p|.
$$
In other words, if $k|p|$ is small, $|f'(z)^p |$ is essentially constant
on boxes, i.e.~
\begin{equation}
  \int_{B_j} | f'(z)^p | \,  \biggl |  \frac{2n_f}{\rho_{\mathbb{H}}}  \biggr |^2 \dzy
  \, \approx  \,
| f'(c_j)^p |  \int_{B_j} \,  \biggl | \frac{2n_f}{\rho_{\mathbb{H}}} \biggr |^2 \dzy,
\end{equation}
where $c_j$ is an arbitrary point in $B_j$.
Hence,
$$
  \int_{B_j} \,  | f'(z)^p | \,  \biggl | \frac{2n_f}{\rho_{\mathbb{H}}} \biggr |^2 \dzy \, \le \, 
\exp(C_2Rk|p|) \cdot M   \int_{B_j} \,  | f'(z)^p | \dzy.
  $$
Summing over all the boxes that make up $A(y)$ gives (\ref{eq:lemma-RHS}).
\end{proof}

The same argument shows:
\begin{lemma}
\label{lemma-RHS2}
Suppose that the average non-linearity of $f \in \mathbf{H}_k^1$ over any box $\square \in \widehat{\mathscr G_n}$ with
$\square \subset [0,1] \times [0, y_0]$ is bounded below by $m$.
Then,\begin{equation}
\label{eq:lemma-RHS2}
u''(y) \ge m \, \exp(-CRk|p|) \cdot \frac{|p|^2}{4y^2} \cdot u(y),
\qquad y \in (0, y_0).
\end{equation}
\end{lemma}

\subsection{Applications}
We now prove Theorem \ref{main-thm} which says that for a fixed $k \in (0,1)$,
\begin{equation}
\label{eq:thm11revisited}
\lim_{p \to 0} \frac{B_k(p)}{|p|^2/4} \, = \,  \Sigma^2(k).
\end{equation}
\begin{proof}[Proof of Theorem \ref{main-thm}]
According to Lemma \ref{boxcart-global}, we may choose  $n \ge 1$ sufficiently large so that the average non-linearity of $f \in \mathbf{H}_k$ over 
any box  in $\widehat{\mathscr G_n}$ is at most $\Sigma^2(k) + \varepsilon/3$. Lemma \ref{lemma-RHS}
implies the differential inequalities
$$u''(y) \le \Bigl (\Sigma^2(k) + 2\varepsilon/3 \Bigr ) \cdot \frac{|p|^2}{4y^2} \cdot u(y), \qquad \text{for }p \text{ small}.$$
Applying  
Lemma  \ref{diff-ineq}(i) gives 
$$u(y) \le C \cdot y^{-|p|^2(\Sigma^2(k)+\varepsilon)/4}.$$
An analogous bound holds for $I_p(y)$ since $I_p(y) \asymp u(y)$. This proves the  upper bound in (\ref{eq:thm11revisited}). 
The lower bound is similar, but uses the converse part of Lemma \ref{boxcart-global} and Lemmas \ref{lemma-RHS2} and \ref{diff-ineq}(ii).
Note that we can assume that $\mu = \bel f$ is supported in the strip $S = \{z \in \Hbar : |\im z| < 1\}$ and is invariant under the translation
$z \to z+1$ to ensure that $f = \tilde w^\mu \in \bH_k^1$ (see the remark following the proof of Lemma \ref{boxcart}).
\end{proof}

 For small $k$, we can give a more precise estimate. Combining Lemma \ref{boxcart} with (\ref{eq:infinitesimal-form}) shows
 that the average non-linearity over a box in $\widehat{\mathscr G_n}$ is bounded by
\begin{equation}
\label{eq:k-small-bound}
\fint_{B_j} \,  \biggl | \frac{2n_f}{\rho_{\mathbb{H}}} \biggr |^2 \dzy \le (\Sigma^2 + C/R) k^2 + Ck^3.
\end{equation}
Putting this into Lemma \ref{lemma-RHS} gives 
\begin{equation}
\label{eq:improved-smirnov2}
u''(y) \le \exp(CRk|p|) \cdot \bigl [ (\Sigma^2 + C/R) k^2 + Ck^3 \bigr ]  \cdot \frac{|p|^2}{4y^2} \cdot u(y).
\end{equation}
Choosing $R \asymp 1/\sqrt{k|p|}$, we get the error term of $\mathcal O \bigl ((k|p|)^{5/2} \bigr)$ in (\ref{eq:improved-smirnov2}). We conclude that for any conformal map $f \in \bH_k^1$,
$$
u(y) \le C \cdot y^{- k^2|p|^2/4 \bigl (\Sigma^2 + C\sqrt{k|p|} \bigr )}.
$$
Repeating the argument for the special conformal maps provided by Lemma \ref{boxcart-global}(ii), we obtain the estimate in the other direction.
As mentioned in the introduction, this entails $D(k) = 1+\Sigma^2 k^2 + \mathcal O(k^{5/2})$.

\input{83}

\input{sparse}

\bibliographystyle{amsplain}

\end{document}

%% file: background.tex
 \section{Background}

\label{sec:background}
In this section, we recall the definition of the universal Teichm\"uller space. We also discuss holomorphic families of conformal maps and 
prove Theorem \ref{sigmak-convex} which says that $\Sigma^2(k)/k^2$ is a convex function of $k \in [0,1)$.

\subsection{Universal Teichm\"uller Space}

The analysis of the universal integral means spectrum $B(p)$ can be thought of as an extremal problem in the universal Teichm\"uller space $\mathcal T(\mathbb{D}^*) = \bigcup_{0\le k < 1} \bSigma_k$. The {\em classical Bers embedding} (with the Schwarzian derivative)
\begin{equation}
\beta_2: \quad
 \varphi \, \to \, s_\varphi \, := \,  \biggl (\frac{\varphi''}{\varphi'} \biggr)' - \frac{1}{2} \biggl (\frac{\varphi''}{\varphi'} \biggr)^2
 \end{equation}
  expresses $\mathcal T(\mathbb{D}^*)$ as a bounded domain
 in a complex Banach space, giving $\mathcal T(\mathbb{D}^*)$ the structure of an infinite-dimensional complex manifold. 
 More precisely, let $A_2^\infty(\mathbb{D}^*)$ be the space of bounded holomorphic quadratic differentials $q(z)dz^2$ equipped with the norm
  $$\| q \|_{A_2^\infty(\mathbb{D}^*)} := \sup_{z \in \mathbb{D}^*} (|z|^2-1)^2 | q(z) |.$$
 Then, as is well-known, 
 the image $\beta_2(\mathcal T(\mathbb{D}^*)) \subset A_2^\infty(\mathbb{D}^*)$ is contained in a ball of radius 6.
 In view of Royden's theorem which equates the Kobayashi and Teichm\"uller metrics on $\mathcal T(\mathbb{D}^*)$, the sets  $\mathbf{\Sigma}_k \subset \mathcal T(\mathbb{D}^*)$
 are metric balls, i.e.~$$\mathbf{\Sigma}_k = \{\varphi: d_T(\id, \varphi) \le d_{\mathbb{D}}(0, k)\}.$$
 This geometric characterization of $\mathbf{\Sigma}_k$ justifies the study of the spectra $B_k(p)$. 
From the point of view of this paper, the Bers embedding (with non-linearity) 
\begin{equation}
\beta_1: \quad
 \varphi \, \to \, n_\varphi \, := \,  \frac{\varphi''}{\varphi'}
\end{equation}
 into the space $A^\infty_1(\mathbb{D}^*)$ of bounded holomorphic 1-forms is more natural.
 Here, the norm of $\phi(z) dz$ is given by
   $$\| \phi \|_{A_1^\infty(\mathbb{D}^*)} := \sup_{z \in \mathbb{D}^*} (|z|^2-1) | \phi(z) |.$$
The image of $\beta_1$ is an open set if we restrict to the subspace of 1-forms that vanish at infinity, i.e.~which have the asymptotics 
 $\phi(z) = \mathcal O(1/|z|^3)$ as $z \to \infty$.

 \begin{remark} (i) The two embeddings present $\mathcal T(\mathbb{D}^*)$ 
 with the same complex structure; however, the metric closures in the ambient Banach spaces are different. The interested reader may consult the work
 of Astala-Gehring \cite{AG} for
 more information.
 
 (ii) If $\bSigma$ is the collection of all conformal maps in principal normalization, then
  $$\inte \beta_2(\bSigma) = \beta_2(\mathcal T(\mathbb{D}^*)) \qquad \text{yet} \qquad \beta_2(\bSigma) \supsetneq \cl \beta_2(\mathcal T(\mathbb{D}^*)).$$ From the fractal approximation principle of Carleson and Jones \cite{CJ}, it is sufficient to take the supremum in the definition of $B(p)$ over $\mathcal T(\mathbb{D}^*)$,
  allowing one to ignore the complement $\bSigma \setminus \mathcal T(\mathbb{D}^*)$. For an alternative approach to fractal approximation, 
  see the work of Beliaev and Smirnov \cite{BS}.
\end{remark}

\subsection{Holomorphic families}

By a {\em holomorphic family} of conformal maps, we mean 
$$
\varphi_t \colon \mathbb{D^*} \to \mathbb{C}, 
\qquad\varphi_0(z) = z, 
\qquad\varphi_t(z) = z + \mathcal O(1/|z|), 
\qquad t \in \mathbb{D}.
$$
According to $\lambda$-lemma, each map $\varphi_t$ admits a $|t|$-quasiconformal extension to the complex plane. Conversely, 
if $\varphi \in \bSigma_k$ has a $k$-quasiconformal extension $H$, then it may be
naturally included into a holomorphic family $\{\varphi_t,\, t \in \mathbb{D}\}$ with $\varphi = \varphi_k$.
This is done by letting $H_t$ be the principal homeomorphic solution to the Beltrami equation $\overline{\partial} H_t = t \, (\mu/k) \, \partial H_t$, 
and then restricting $H_t$ to the exterior unit disk.

\begin{lemma}
Given a standard holomorphic family $\{\varphi_t = w^{t\mu}, \, t \in \mathbb{D}\}$ with $|\mu| \le \chi_{\mathbb{D}}$, the map $t \to b_{\varphi_t} =  \log\varphi_t'$ is a Banach-valued holomorphic function from $\mathbb{D}$ to the Bloch space of the exterior
unit disk.
In particular, the mappings $t \to n_{\varphi_t} \in A^\infty_1(\mathbb{D}^*)$ and $t \to s_{\varphi_t} \in A^\infty_2(\mathbb{D}^*)$ are holomorphic.
\end{lemma}

\begin{proof}[Proof]
The holomorphy of $n_{\varphi_t}$ and $s_{\varphi_t}$ follows from the boundedness
of $b \to b'$, $\mathcal B(\mathbb{D}^*) \to A^\infty_1(\mathbb{D}^*)$ and $\phi \to \phi' - \frac{1}{2} \cdot \phi^2$, $A^\infty_1(\mathbb{D}^*) \to A^\infty_2(\mathbb{D}^*)$. To see that $b_{\varphi_t}$ is holomorphic, it suffices to check that it
is weak-$*$ holomorphic.

For simplicity of exposition, let us instead show that any norm-bounded, pointwise holomorphic function from the unit disk into $\mathcal B$ is weak-$*$ holomorphic. 
As is well-known \cite{hedenmalm, zhu-spaces}, the predual of $\mathcal B$ is the Bergman space $A^1$, with the dual pairing
$$\langle b, g \rangle =  \lim_{r \to 1} \frac{1}{\pi} \int_{\mathbb{D}} b(z) \overline{g(rz)} \, |dz|^2, \qquad
b \in \mathcal B, \ g \in A^1.$$
Since the dilates $g_r(z) = g(rz)$ converge to $g(z)$ in $A^1$, the above limit converges uniformly in $r$
as $b$ ranges over bounded subsets of the Bloch space. Hence,
$$
t \, \to \, \langle b_t, g \rangle \, = \, \lim_{r \to 1} \frac{1}{\pi} \int_{\mathbb{D}} b_t(z) \overline{g(rz)} \, |dz|^2
$$
is a holomorphic function, being the uniform limit of holomorphic functions.
\end{proof}

From the Neumann series expansion for principal solutions to the Beltrami equation \cite[p.~165]{AIM},
\begin{equation} \label{neumann}
\varphi'_t = \partial  \varphi_t = 1+t \mathcal{S}\mu+t^2 \mathcal{S}\mu \mathcal{S}\mu+ \ldots, \qquad |z|>1,
\end{equation}
 it follows that the derivative of the Bers embedding at the origin is just 
\begin{equation}
\label{eq:inf-form1}
(d/dt) \bigl |_{t=0} \, \log \varphi_t' = \mathcal S\mu.
\end{equation}
 In particular, $\mathcal S\mu \in \mathcal B(\mathbb{D}^*)$ and 
\begin{equation}
\label{eq:infinitesimal-form}
\biggl \| \frac{ \log\varphi_t'}{t} - \mathcal S\mu \biggr\|_{\mathcal B(\mathbb{D}^*)} = \mathcal O(|t|), \quad \text{for }|t| < 1/2.
\end{equation}
 Since the asymptotic variance is continuous in the Bloch norm \cite{hedenmalm-atvar}, the function
\begin{equation}
u(t) = \sigma^2 \biggl ( \frac{\log\varphi_t'}{t}  \biggr )
\end{equation} extends continuously to $\sigma^2(\mathcal S\mu)$ at $t=0$. 
Similarly to (\ref{eq:inf-form1}),
$(\mathcal S\mu)' = (d/dt) \bigl |_{t=0} \, n_{\varphi_t}$ and
$(\mathcal S\mu)'' = (d/dt) \bigl |_{t=0} \, s_{\varphi_t}$ 
are the infinitesimal forms of the non-linearity and the Schwarzian derivative respectively.

\begin{proof}[Proof of Theorem \ref{sigmak-convex}]
Taking the supremum of (\ref{eq:infinitesimal-form}) over all $|\mu| \le \chi_{\mathbb{D}}$ shows that $\lim_{k \to 0} \Sigma^2(k)/k^2 = \Sigma^2$, which is 
the statement (ii).

Part (i) uses a  fractal approximation argument.
According to Theorem \ref{AIPP-fat}, in the definition
$\Sigma^2(k) = \sup_{\varphi \in \bSigma_k} \sigma^2(\log \varphi')$, it suffices to take the supremum over conformal maps
$\varphi = w^{\mu}$ that have ``dynamically-invariant'' quasiconformal extensions.
According to  \cite[Section 8]{AIPP}, in these fractal cases, the function $u(t)$ is subharmonic. In particular, this implies that
$\sup_{|t|=k} u(t)$
is a non-decreasing convex function of $k \in [0,1)$. Therefore, $\Sigma^2(k)/k^2$ is also a non-decreasing convex function, being the
supremum of such functions.
\end{proof}

%% file: qcdimbulk.tex
\newpage

\section{Working on the upper half-plane}

\label{sec:life-on-H}

Since non-linearity is not invariant under $\aut(\mathbb{D})$, it is convenient to work on the upper half-plane. This
makes the computations quite a bit simpler. If $b \in \mathcal B(\mathbb{H})$ is a holomorphic function on $\mathbb{H}$ with 
\begin{equation}
\label{eq:def-bloch}
\|b\|_{\mathcal B(\mathbb{H})} = \sup_{z \in \mathbb{H}}\, 2y \cdot |b'(z)| \le \infty,
\end{equation}
 we define its {\em asymptotic variance} as
\begin{equation}
\sigma^2_{[0,1]}(b)  = \limsup_{y \to 0^+} \, \frac{1}{|\log y|}  \int_0^1 |b(x+iy)|^2 \, dx.
 \end{equation}
In \cite[Section 6]{mcmullen}, McMullen  showed that one can compute the asymptotic variance 
of Bloch functions by examining C\'esaro averages of integral means that
involve
higher order derivatives. 
Here we shall be content with the formula
\begin{equation}
\label{eq:mcm-ca}
\sigma^2_{[0,1]}(b)  = \limsup_{h \to 0^+} \, \frac{1}{|\log h|} \int_h^1 \int_0^1 \biggl |\frac{2b'(x+iy)}{\rhoH}\biggr |^2 \dzy.
 \end{equation}
 
 Let  $\bH$ denote the class  
of conformal maps $f: \mathbb{H} \to \mathbb{C}$ which fix the points $0,1,\infty$ and
$\bH_k \subset \bH$ be the subclass consisting of conformal maps that admit a $k$-quasiconformal extension
to the lower half-plane $\Hbar$.
For $f \in \bH$, the  {\em integral means spectrum} is given by
 $$
\beta_f(p) = \limsup_{y \to 0^+} \frac{\log \int_0^1 |f'(x+iy)^p| \, dx}{|\log y|}, \qquad p \in \mathbb{C}.
 $$

For a Beltrami coefficient $\mu$ supported on $\Hbar$  with $\| \mu \|_\infty < 1$, let
$\tilde w^\mu \in \bH_k$ denote the {\em normalized} solution  of the Beltrami equation $\overline{\partial} w = \mu\, \partial w$.
(The notation $w^\mu$ is reserved for principal solutions defined for compactly supported coefficients.)

Since the formula for the Beurling transform
(\ref{eq:beurling}) may not converge if $\mu$ is not compactly supported, we are obliged to work with a modified Beurling transform
\begin{equation}
\label{eq:beurling22}
 \mathcal S^\# \mu(z) \, =\, -\frac{1}{\pi} \int_{\Hbar} \mu(\zeta) \biggl [ \frac{1}{(\zeta-z)^2} - \frac{1}{\zeta^2} \biggr ] \, |d\zeta|^2,
 \end{equation} 
however, the formula for the derivative remains the same:
\begin{equation}
\label{eq:beurling23}
 \text{``}(\mathcal  S \mu)'(z)\text{''} \, := \,  (\mathcal S^\# \mu)'(z)  \, =\, -\frac{2}{\pi} \int_{\Hbar} \frac{\mu(\zeta)}{(\zeta-z)^3} \, |d\zeta|^2.
 \end{equation}
Not surprisingly,
$\mathcal S^\# \mu$ and $\log f'$ are Bloch functions. In fact,
\begin{equation}
\| \mathcal S^\# \mu \|_{\mathcal B(\mathbb{H})} \le \frac{8}{\pi} \cdot \| \mu \|_\infty, \qquad
\| \log  f' \|_{\mathcal B(\mathbb{H})} \le 6k,
\end{equation}
have the same bounds as do Bloch functions on the disk. Furthermore, the universal bounds are also unchanged:
\begin{lemma}
\label{life-on-H}
$$\Sigma^2 = \sup_{|\mu| \le \chi_{\Hbar}} \sigma^2_{[0,1]}(\mathcal S^\# \mu), \qquad
\Sigma^2(k) =  \sup_{f \in \bH_k}  \sigma^2_{[0,1]}(\log f'),
$$
$$
B_k(p) = \sup_{f \in \bH_k} \beta_f (p).
$$
\end{lemma}

\subsection{Exponential transform}

A convenient way to transfer results from the half-plane to the disk is by exponentiating.
Let $\xi(w) = e^{-2\pi i w}$ be  the exponential mapping which takes $\mathbb{H} \to \mathbb{D}^*$. 
For a normalized $k$-quasiconformal mapping $f$, define its {\em exponential transform} as
\begin{equation}
\label{eq:exponential-transform}
 \mathcal E_f(w) = -\frac{1}{2\pi i} \cdot \log \circ f \circ \xi(w),
\end{equation}
where the branch of logarithm is chosen so that $\log  \mathcal E_f(1) = 1$.
In terms of Beltrami coefficients, the dilatation $\bel  \mathcal E_f = \xi^* (\bel f)$ is just the pullback of $\bel f$, considered as a
 $(-1,1)$ form.

It is not difficult to see that $f \to  \mathcal E_f(w)$ establishes a bijection between normalized $k$-quasiconformal mappings with ones
satisfying the periodicity condition
\begin{equation}
\mathcal E_f(w+1)=\mathcal E_f(w)+1.
\end{equation}
Of interest to us, \ $\varphi$ is conformal on $\mathbb{D}^*\ \Longleftrightarrow \ \mathcal E_\varphi$ is conformal on $\mathbb{H}$.

\begin{lemma}
\label{exponential-lemma}
{\em (i)} If $|\mu| \le \chi_{\mathbb{D}}$, then for $w \in \mathbb{H}$,
\begin{equation}
\label{eq:error-beurling}
\Biggl | \, \biggl |\frac{(\mathcal S (\xi^* \mu))'}{\rho_{\mathbb H}}(w) \biggr | \, - \, \biggl | \frac{(\mathcal S\mu)'}{\rho_*}(\xi(w)) \biggr| \, \Biggr | = o(1),
\quad
\text{as } \im w \to 0.
\end{equation}
\noindent
{\em (ii)} If $\varphi$ is a normalized $k$-quasiconformal mapping that is conformal on $\mathbb{D}^*$, then
\begin{equation}
\label{eq:error-nonlinearity}
\Biggl | \, \biggl |\frac{n_{\mathcal E_\varphi}}{\rho_{\mathbb H}}(w) \biggr | \, - \, \biggl | \frac{n_\varphi}{\rho_*}(\xi(w)) \biggr| \, \Biggr | = o(1),
\quad
\text{as } \im w \to 0.
\end{equation}
\end{lemma}

The above lemma follows from two observations about the exponential which imply that is does not change the asymptotic features of non-linearity:

\smallskip
 
I. $\xi$ is approximately a local isometry on $\{w \in \mathbb{H}: \im w < \delta \}$, when $\mathbb{H}$ and $\mathbb{D}^*$ are equipped with
their hyperbolic metrics.

II. $\xi$ is approximately linear on small balls $B(x, \delta)$ with $x \in \mathbb{R}$:
$$
\biggl | \frac{1}{\xi'(x)} \cdot \frac{\xi(z) - \xi(w)}{z-w} - 1 \biggr | < \varepsilon, \qquad z, w \in B(x, \delta).
$$
Both of these properties are consequences of Koebe's distortion theorem, where one takes 
into account that the exponential maps the real axis to the unit circle. The reader can consult \cite[Section 2]{ivrii-phd} for more details.

From Lemma \ref{exponential-lemma}, it is clear that
 $$\sigma^2(b) = \sigma^2_{[0,1]}(\xi^*\mu)
\qquad \text{and} \qquad \beta_\varphi(p) = \beta_{\mathcal E_\varphi}(p),$$

The reader versed in the arguments of Sections \ref{sec:inf-locality}--\ref{sec:gl-locality} should have no trouble filling in the details
in Lemmas \ref{life-on-H} and \ref{exponential-lemma}.

\subsection{Boxes and grids}

 By a {\em box} in the upper half-plane, we mean a rectangle 
whose sides are parallel to the coordinate axes, with the bottom side located above the real axis. 
We say two boxes
are {\em similar} if they differ by an affine scaling $L(z) = az + b$ with $a > 0$, $b \in \mathbb{R}$.
We use $\overline{B} \subset \overline{\mathbb{H}}$ to denote the reflection of the box $B$ with respect to the real line.
Every box $B$ is similar to a unique box of the form $[0, \alpha] \times [1/n, 1]$. In this case, we
 say that $B$ is of type $(n, \alpha)$.

Boxes naturally arise in grids. By a {\em grid}, we mean a collection of similar boxes that tile $\mathbb{H}$.
One natural collection of grids are the {\em $n$-adic grids} \,$\mathscr G_n$, defined for integer $n \ge 2$. An {\em $n$-adic interval}
is an interval of the form $I_{j,k} = \bigl [j \cdot n^{-k}, (j+1) \cdot n^{-k} \bigr ]$. To an $n$-adic interval $I$, we associate
the {\em $n$-adic box} $$\square_I \, = \, \Bigl \{ w : \re w \in I, \, \im w \in \bigl [n^{-1}|I|, \,|I| \bigr ] \Bigr \}.
$$
It is easy to see from the construction that the boxes $\square_{I_{j,k}}$, with $j, k \in \mathbb{Z}$ have disjoint interiors and their union is $\mathbb{H}$.

Given two boxes $B_1$ and $B_2$, we say $B_1$ {\em dominates} $B_2$ if $B_1 = [x_1, x_2] \times [y_1, y_2]$ and $B_2 = [x_1, x_2] \times [\theta y_1, \theta y_2]$ for some $0 < \theta \le 1$. In other words, $B_2$ has the same hyperbolic height and is located strictly beneath $B_1$.
We let $\widehat{\mathscr G_n}$ denote the collection of boxes that are dominated by some box in $\mathscr G_n$ with $1/n < \theta \le 1$.
The advantage of the collection $\widehat{\mathscr G_n}$ is that any horizontal strip 
$\mathbb{R} \times [y/n,y] \subset \mathbb{H}$ of hyperbolic height $\log n$ can be tiled by boxes in $\widehat{\mathscr G_n}$.
This property will be used in Section \ref{sec:integral-means}.

Suppose $\mu$ is a Beltrami coefficient supported on the lower half-plane.
We say that $\mu$ is {\em periodic} with respect to a grid $\mathscr G$ (or rather with respect to $\overline{\mathscr G}$) if for any two 
boxes $B_1, B_2 \in \mathscr G$, we have
$
\mu|_{\overline{B}_1} = L^*(\mu|_{\overline{B}_2})$, where $L$ is the affine map that takes $B_1$ to $B_2$.
Given $\mu$ defined on a box $B$, and a grid $\mathscr G$ containing $B$, there exists a unique periodic Beltrami coefficient $\mu_{\per}$
which agrees with $\mu$ on $B$.

\section[Locality of Beurling transform]{Locality of $(\mathcal S\mu)'/\rho_{\mathbb{H}}$}
\label{sec:inf-locality}

The technique of {\em fractal approximation} from \cite{AIPP} hinges on the local nature of the 
 operator $\mu \to \mathcal (\mathcal S\mu)'$. We first show that
$(\mathcal S\mu)'(z)$ is bounded as a 1-form:

\begin{lemma}
\label{qbounds-H}
Suppose $\mu$ is a Beltrami coefficient supported on the lower half-plane with $\|\mu\|_\infty \le 1$. Then,
for $z \in \mathbb{H}$, $\bigl |(2(\mathcal S\mu)'/\rhoH)(z) \bigr | \le 8/\pi$.
\end{lemma}

\begin{proof}The proof is by direct calculation:
\allowdisplaybreaks
\begin{align*}
| (\mathcal S\mu)'(z)| & \le  \frac{2}{\pi} \int_{\Hbar}  \frac{1}{|w-z|^3} \, | dw|^2,\\
& =   \frac{2}{\pi} \int_{y_0}^\infty \int_{-\infty}^\infty  \frac{1}{(x^2+y^2)^{3/2}} \, dxdy, \\
& =   \frac{2}{\pi} \int_{y_0}^\infty \frac{x}{y^2 \sqrt{x^2+y^2}} \biggl |_{-\infty}^\infty \, dy, \\
& =   \frac{4}{\pi} \int_{y_0}^\infty \frac{dy}{y^2},\\
& =   \frac{4}{\pi y_0},
\end{align*}
where $y_0 = \im z$. Multiplying by 2 and dividing by the density of the hyperbolic metric gives the result.
\end{proof}

If $z \in \mathbb{H}$ is far away from the reflection of the support of $\mu$, one can give a better estimate. 
For a point $x + iy \in \mathbb{H}$, define its ``square'' neighbourhood as
$$
Q_L(x + iy) := \bigl \{ w : \re w \in [x - e^L y, x + e^L y], \, \im w \in [e^{-L} y, e^L y] \bigr \},
$$
and let $Q_L(x-iy)$ be its reflection in the lower half-plane.
\begin{lemma}
\label{square-locality}
Under the assumptions of Lemma \ref{qbounds-H}, if $\mu = 0$ on $Q_L(x - iy)$, then 
\begin{equation}
\label{eq:square-locality}
\biggl |\frac{2(\mathcal S\mu)'}{\rhoH} (x+iy)\biggr | \le Ce^{-L}.
\end{equation}
\end{lemma}

\begin{proof}
The lemma follows by estimating the contributions of the top, bottom, left and right parts of $\Hbar \setminus Q_L(x - iy)$ separately and adding them up. We leave the details to the reader.
\end{proof}

 Lemma \ref{square-locality} says that to determine the value of  $(\mathcal S \mu)'/\rho_{\mathbb{H}}$ at a point $z \in \mathbb{H}$, up to small error,
 it suffices to know the values of $\mu$ in a neighbourhood of $\bar z$.
In particular, if $\mu_1$ and $\mu_2$ are two Beltrami coefficients that agree on $Q_L(x-iy)$ with $L$ large, then the difference
$|((\mathcal S\mu_1 - \mathcal S\mu_2)'/\rho_{\mathbb{H}})(z)|$
is small.

\begin{remark}
It may seem more natural to work with round neighbourhoods of $x-iy$, i.e. to assume that $\mu$ vanishes on
 $\{w \in \Hbar : d_{\Hbar}(w, x-iy) < L\}$. However, in this
case, the estimate (\ref{eq:square-locality}) is only $\le CLe^{-L}$.
\end{remark}

We now come to the main result of this section.

\begin{lemma}
\label{boxcart5}
Suppose $\mu_1$ and $\mu_2$ are two Beltrami coefficients on $\overline{\mathbb{H}}$ with $\|\mu_i\|_\infty \le 1$, $i=1,2$. If $\mu_1 = \mu_2$ agree on 
an $(n, \alpha)$-box $\overline{B}$ with $\alpha \ge 1$, then
 \begin{equation}
 \label{eq:boxcart5}
\Biggl | \,
 \fint_B \, \biggl |\frac{2(\mathcal S\mu_1)'}{\rho_{\mathbb H}}(z) \biggr |^2 \dzy \, - \,
  \fint_B \, \biggl |\frac{2(\mathcal S\mu_2)'}{\rho_{\mathbb H}}(z) \biggr |^2 \dzy \, \Biggr | \, \le \, \frac{C_1}{\log n}.
  \end{equation}
\end{lemma}

\begin{proof}
Without loss of generality, we may assume that $B = [0,\alpha] \times [1/n,1]$.
From the elementary identity $\bigl | |a|^2 -|b| ^2 \bigr| = |a-b| \cdot |a+b|$ and Lemma \ref{qbounds-H}, it follows that the left hand side of (\ref{eq:boxcart5}) is bounded by
\begin{equation}
 \label{eq:boxcart6}
\frac{32}{\pi} \fint_{ B} \, \biggl |\frac{\bigl (\mathcal S(\mu_1-\mu_2) \bigr)'}{\rhoH}(z) \biggr |  \dzy.
\end{equation}
 The claim now follows from  Lemma \ref{square-locality} and integration.
\end{proof}

\section{Box Lemma}
\label{sec:box}

In this section, we show the infinitesimal version of the box lemma:

\begin{lemma}
\label{boxcart}
{\em (i)} 
For any Beltrami coefficient $\mu$ with $|\mu| \le \chi_\mathbb{\overline{H}}$ and
 $(n, \alpha)$-box $\square \subset \mathbb{H}$ with $\alpha \ge 1$, 
 \begin{equation}
 \label{eq:boxcart}
 \fint_{\square} \, \biggl |\frac{2(\mathcal S\mu)'}{\rhoH}(z) \biggr |^2 \, \frac{|dz|^2}{y} \, < \, \Sigma^2 + \frac{C}{\log n}.
 \end{equation}
 
 {\em (ii)} Conversely, for $n \ge 1$, there exists a Beltrami coefficient $\mu$, periodic with respect to the $n$-adic grid, which satisfies
 \begin{equation}
 \label{eq:boxcart-local-converse}
 \fint_{\square} \, \biggl |\frac{2(\mathcal S\mu)'}{\rho_{\mathbb{H}}}(z) \biggr |^2 \, \frac{|dz|^2}{y}  \, > \, \Sigma^2 - \frac{C}{\log n}
 \end{equation}
 on every box $\square \in \widehat{\mathscr G_{n}}$.
\end{lemma}

\begin{proof}
(i) Assume for the sake of contradiction that there is a box $\square \subset \mathbb{H}$ and a Beltrami coefficient $\mu$ for which
 \begin{equation}
 \fint_{\square} \, \biggl |\frac{2(\mathcal S\mu)'}{\rhoH}(z) \biggr |^2 \dzy \, > \, \Sigma^2 + \frac{C}{\log n},
 \end{equation}
 with $C > C_1$ from  Lemma \ref{boxcart5}.  Let $\mathscr G = \bigcup_{j=1}^\infty \overline{\square}_j$ be a grid containing $\overline{\square}$ and form the Beltrami coefficient $\mu_{\per}$ by restricting $\mu$ to $\overline{\square}$
 and periodizing with respect to $\mathscr G$.
According to Lemma \ref{boxcart5}, we would have
 \begin{equation}
 \label{eq:boxcart7}
 \fint_{\square_j} \, \biggl |\frac{2(\mathcal S\mu_{\per})'}{\rhoH}(z) \biggr |^2 \dzy \, > \, 
 \Sigma^2 + \varepsilon, \quad  \text{for all }\, \square_j \in \mathscr G.
  \end{equation}
In view of (\ref{eq:mcm-ca}), this implies $\sigma^2_{[0,1]}(\mathcal S^\#\mu_{\per}) > \Sigma^2 + \varepsilon$, which contradicts the definition of $\Sigma^2$.
  
(ii)
Conversely, suppose $\nu$ is a Beltrami coefficient with $$|\nu| \le \chi_{\overline{\mathbb{H}}}
\quad \text{and} \quad \sigma^2_{[0,1]}(\mathcal S^\#\nu) \ge \Sigma^2 - \varepsilon.$$ 
Consider the $n$-adic grid $\mathscr G_{n}$.
By the pigeon-hole principle, there exists an $n$-adic box $\square$
for which the integral in (\ref{eq:boxcart}) is at least $\Sigma^2 - \varepsilon$.
Restricting $\nu$ to $\square$ and periodizing over $n$-adic boxes produces a Beltrami coefficient $\nu_{\per}$ which 
satisfies
\begin{equation}
\label{eq:std-box}
 \fint_{\square_j} \, \biggl |\frac{2(\mathcal S\nu_{\per})'}{\rho_{\mathbb{H}}}(z) \biggr |^2 \, \frac{|dz|^2}{y}  \, > \, \Sigma^2 - \frac{C}{\log n},  
 \quad \text{for all }\, \square_j \in \mathscr G_{n}.
\end{equation}
A careful inspection reveals that the above estimate holds for all boxes in the collection $\widehat{\mathscr G_{n}}$.
This completes the proof of Lemma \ref{boxcart}.
\end{proof}

\begin{remark}
Given a periodic Beltrami coefficient $\mu_{\per}$ on $\Hbar$ from (ii), we may multiply it by
the characteristic function of the strip $S = \{z \in \Hbar : |\im z| < 1\}$
to get a Beltrami coefficient $\mu_{\per} \cdot \chi_{S}$ on $\Hbar$ that is periodic under $z \to z+1$.
By construction, $\mu_{\per} \cdot \chi_{S}$ descends
to a Beltrami coefficient on the unit disk  via the exponential mapping, which is eventually-invariant under $z \to z^n$. Clearly, the asymptotic variance is unchanged in this process.
This proves the statement (\ref{eq:AIPP-fat}) from Theorem \ref{AIPP-fat}.
\end{remark}

\section[Locality of non-linearity]{Locality of $n_f/\rho_{\mathbb{H}}$}
\label{sec:gl-locality}


\begin{lemma} 
\label{boxcart-global}
{\em (i)} Fix $0 < k < 1$.
Given $\varepsilon > 0$, if $n \ge n(\varepsilon)$ is sufficiently large, then for any  $(n, \alpha)$-box $\square \subset \mathbb{H}$ with $\alpha \ge 1$ 
and any conformal map $f \in \mathbf{H}_k$,
 \begin{equation}
 \label{eq:boxcart2}
 \fint_{\square} \, \biggl |\frac{2n_f}{\rho_{\mathbb{H}}}(z) \biggr |^2 \, \frac{|dz|^2}{y} \, < \, \Sigma^2(k) + \varepsilon.
 \end{equation}
 
 {\em (ii)} Conversely, for any $\varepsilon > 0$, there exists a conformal map $f = \tilde w^\mu \in \mathbf{H}_k$, whose dilatation 
 $\mu = \bel f := \overline{\partial} f/\partial f$ is periodic with respect
to the $n$-adic grid for some $n \ge 1$,
and which satisfies
  \begin{equation}
 \label{eq:boxcart3}
 \fint_{\square} \, \biggl |\frac{2n_f}{\rho_{\mathbb{H}}}(z) \biggr |^2 \, \frac{|dz|^2}{y} \, > \, \Sigma^2(k) - \varepsilon
 \end{equation}
  on every box $\square \in \widehat{\mathscr G_{n}}$.
\end{lemma}

Since we do not require a quantitative estimate, it suffices to give a simple compactness argument. In view of the arguments from the previous section,
it is enough to show:
\begin{lemma}
\label{compactness-argument}
Suppose $\tilde w^{\mu_1}$ and $\tilde w^{\mu_2} \in \bH_k$ are two $k$-quasiconformal mappings that are conformal on the upper half-plane.
For any $\varepsilon > 0$, there exists $R$ sufficiently large so that if 
$\mu_1 = \mu_2$
on $B_{\hyp}(\overline{w}_0, R) = \{ w \in \Hbar : d_{\mathbb{H}}(w, \overline{w}_0) > R\}$, $w_0 \in \mathbb{H}$, then
\begin{equation}
\label{eq:compactness-argument}
\biggl |\frac{n_{\tilde w^{\mu_1}} - n_{\tilde w^{\mu_2}}}{\rho_{\mathbb{H}}} (w_0)\biggr |< \varepsilon.
\end{equation}
\end{lemma}

\begin{proof}
Since non-linearity is invariant under compositions with affine maps $z \to az +b$, $a>0$, $b \in \mathbb{R}$,
 it suffices to prove the lemma with $w_0 = i$.

To the contrary, suppose that one could find sequences of Beltrami coefficients $\{\mu_n\}$ and $\{\nu_n\}$, 
with $\mu_n = \nu_n$
on $B_n = B_{\hyp}(-i, n)$ and 
\begin{equation}
\label{eq:compactness-argument2}
\|\tilde w^{\mu_n} - \tilde w^{\nu_n}\|_{L^\infty(B)} > 1/n, \quad \text{where }B = \{w \in \mathbb{H}  : d_{\mathbb{H}}(w, i) < 1\}.
\end{equation}
 Since the collection of normalized quasiconformal mappings with dilatation bounded by $k$ forms a normal family, we can extract 
 a convergent subsequence
$\tilde w^{\mu_{n_j}} \to \tilde w$. 

Stoilow factorization allows us to write $\tilde w^{\nu_n} = H_n \circ \tilde w^{\mu_n}$, where $H_n$ is a normalized quasiconformal map
with $\supp(\bel {H_n}) \subseteq w^{\mu_n}(\mathbb{H} \setminus B_n)$ and $\|\bel {H_n}\|_\infty < \frac{2k}{1+k^2}$.
Since the supports of $\bel {H_{n_j}}$ shrink to $\tilde w(\mathbb{R})$ which has measure 0, the only possible limit of $H_{n_j}$ is the identity.
This rules out (\ref{eq:compactness-argument2}), thus proving the lemma.
\end{proof}

\begin{remark}
Using the methods of \cite[Section 5]{BJ}, one can show that (\ref{eq:compactness-argument}) decays exponentially in $R$, i.e.~$\lesssim e^{-\gamma R}$ for some $\gamma > 0$.
\end{remark}

%% file: 83.tex
\section[The improved error term]{The $\mathcal O(k^{8/3-\varepsilon})$ error term}

In this section,
 we show that the maximal Hausdorff dimension of a $k$-quasicircle satisfies
\begin{equation}
\label{eq:83}
D(k) = 1 + \Sigma^2 k^2 + \mathcal O(k^{8/3-\varepsilon}), \qquad \text{for any } \varepsilon > 0.
\end{equation}
We focus on the upper bound in (\ref{eq:83}) and leave the lower bound to the reader.
Suppose $f: \mathbb{H} \to \mathbb{C}$ is a conformal mapping of the upper half-plane which admits a $k$-quasiconformal extension with $0 < k <1/2$.
Let $$
B =  [0, 1] \times \bigl [1/e^R, 1 \bigr], \qquad R = k^{-\gamma}, \quad \gamma > 0,
$$
be a box in $\mathbb{H}$ and $z_B$ be the midpoint of its top edge. 
Our objective is to slightly improve the argument of Lemma \ref{lemma-RHS} by showing that:
\begin{lemma} Suppose $p \in [1,2)$. For any $\gamma \in (0,2/3)$,
\label{quotient-lemma}
\begin{equation}
\label{eq:quotient}
\mathcal Q(B) \, := \,
 \frac{  \int_{B} | f'(z)^p | \,  \bigl |  \frac{2n_f}{\rho_{\mathbb{H}}}  \bigr |^2 \dzy}{  \int_{B} | f'(z)^p | \dzy} 
 \, \le \,
  \Sigma^2 k^2 + \mathcal O_\gamma( k^{2+\gamma} ).
\end{equation}
\end{lemma}
From the scale-invariance of the problem, the same estimate must also hold on any box similar to $B$.
The arguments of the previous section now give the upper bound in (\ref{eq:83}). 
Here, we remind the reader that in view of (\ref{eq:beta-dimension}), the range of exponents $p \in [1,2)$ is sufficient for applications to Minkowski dimension.

In the proof of  Lemma \ref{lemma-RHS}, we made the assumption that $Rkp$ was small in order to guarantee that $|f'(z)^p|$ was approximately constant in $B$.
To be able to take $R = k^{-\gamma}$ with $\gamma > 1/2$, we introduce the {\em exceptional set}
\begin{equation}
\label{eq:exceptional-set}
\mathscr E := \bigl \{z \in B: |\log f'(z) - \log f'(z_B)| > k^{\gamma} \bigr  \}.
\end{equation}

\begin{lemma}
\label{global-bound}
Suppose $p \in [1,2)$ and $\gamma \in (0,2/3)$. Then,
\begin{equation}
\label{eq:exc-areabound}
\int_{\mathscr E} \dzy < C \cdot k^{\gamma} \cdot \int_{B} \dzy
\end{equation}
and
\begin{equation}
\label{eq:exc-areabound2}
\int_{\mathscr E} | f'(z)^p | \dzy < C \cdot k^{\gamma} \cdot | f'(z_B)^p | \int_{B} \dzy.
\end{equation}
\end{lemma}
 With the above lemma, the proof of Lemma \ref{quotient-lemma} runs as follows:

\begin{proof}[Proof of Lemma \ref{quotient-lemma}  assuming Lemma \ref{global-bound}.]
\ Since $$\biggl |  \frac{2n_f}{\rho_{\mathbb{H}}}  \biggr |^2 < 36k^2 \qquad \text{and} \qquad \osc_{B} |f'(z)^p| < 2,$$ the  bound
(\ref{eq:exc-areabound2}) gives 
$$
\mathcal Q(B) \, = \, \frac{  \int_{B} | f'(z)^p | \,  \bigl |  \frac{2n_f}{\rho_{\mathbb{H}}}  \bigr |^2 \dzy}{  \int_{B} | f'(z)^p | \dzy} 
\, = \, \frac{  \int_{B \setminus \mathscr E} | f'(z)^p | \,  \bigl |  \frac{2n_f}{\rho_{\mathbb{H}}}  \bigr |^2 \dzy}{  \int_{B \setminus \mathscr E} | f'(z)^p | \dzy}
+ \mathcal O(k^{2+\gamma}).
$$
From the definition of the exceptional set (\ref{eq:exceptional-set}),
$$
 1 - Ck^\gamma \, \le \, \biggl | \frac{f'(z)}{f'(z_B)} \biggr|^p  \, \le \, 1 + Ck^\gamma, \qquad z \in B \setminus \mathscr E,
$$
we obtain
$$
\mathcal Q(B) = \fint_{B \setminus \mathscr E} \,  \biggl | \frac{2n_f}{\rho_{\mathbb{H}}} \biggr |^2 \dzy + \mathcal O(k^{2+\gamma}).
$$
According to (\ref{eq:k-small-bound}), the average non-linearity over $B$ is bounded by $(\Sigma^2 + Ck^\gamma) k^2$. Combining with
 (\ref{eq:exc-areabound}) shows that the average non-linearity over $B \setminus \mathscr E$ is also at most
 $(\Sigma^2 + Ck^\gamma) k^2$. 
This completes the proof.
\end{proof}

Let $L_S = \{z \in B : \im z = 1/e^S\}$ be the line segment consisting of points in $B$ for which the hyperbolic distance to the top edge  is $S$ and define $\mathscr E_S := \mathscr E \cap L_S$.
Set $g(z) = \log f'(z) - \log f'(z_B).$
Since $\| \log f' \|_{\mathcal B(\mathbb{H})} \le 6 k$, 
\begin{equation}
\label{eq:hawk-bound}
\bigl | \bigl \{ z \in L_S: |\re g(z)| > \eta \bigr \} \bigr | \le \exp \biggl ( - c_0 \cdot \frac{\eta^2}{k^2 S} \biggr ),
\end{equation}
where $c_0 > 0$ is a universal constant. This follows from the sub-Gaussian estimate for martingales with bounded increments 
 \cite[Equation (2.9)]{makarov90}. Alternatively, for an analytic proof of (\ref{eq:hawk-bound}), the reader may consult \cite{hedenmalm-atvar}.

\begin{proof}[Proof of Lemma \ref{global-bound}]
To show (\ref{eq:exc-areabound}), it  suffices to demonstrate that 
$|{\mathscr E_S}| < C \cdot k^{\gamma}$ for any $0 \le S \le R$. 
Putting $\eta = k^\gamma$ and $S \le R = k^{-\gamma}$ in (\ref{eq:hawk-bound}) gives
\begin{equation}
|{\mathscr E_S}|  \le \exp (-c_0 \cdot k^{3\gamma-2}).
\end{equation} In order to get any kind of decay, we must have $3\gamma - 2 < 0$. Of course, any $\gamma < 2/3$ gives
an {\em exponential} decay rate, which is more than sufficient.
The second statement (\ref{eq:exc-areabound2}) follows from (\ref{eq:exc-areabound}) and the fact that $\osc_{B} |f'(z)^p| < 2$.
\end{proof}

%% file: sparse.tex
\section{Sparse Beltrami Coefficients}

In this section, we prove Theorem \ref{sparse-thm} which gives a stronger  bound for the dimension of quasicircle $\tilde w^\mu([0,1])$ 
if the dilatation $\mu$ has small support.
We assume that $\mu$ is a Beltrami coefficient on $\Hbar$ for which $\supp \mu \subseteq \mathcal G = \bigcup_{j \in J} A_j$ 
is a union of ``crescents'' that satisfy the quasigeodesic and separation properties. That is, we assume 
each crescent lies within a hyperbolic distance $S$ from a geodesic arc $\gamma_j \subset \Hbar$ and that the hyperbolic distance between any two crescents is bounded below by $R$.

We are interested in studying the quadratic behaviour of $t \to \Mdim \tilde w^{t\mu}([0,1])$ when $S$ is fixed and $R$ is large. In the dynamical setting,
this was considered in \cite{ivrii-phd}, where it was sufficient to measure the length of the intersection of $\mathcal G$ with
 horizontal lines. The non-dynamical case was examined by Bishop \cite{bishop-bigdef}, but without the quadratic behaviour.
The proof of Theorem \ref{sparse-thm} follows from the Becker-Pommerenke argument and the estimate:

\begin{theorem}
\label{boxcart-sparse}
For any  Beltrami coefficient $\mu$ with $\|\mu\|_\infty \le 1$ and $\supp \mu \subseteq \mathcal G$, we have
 \begin{equation}
 \fint_{L} \, \biggl |\frac{2(\mathcal S\mu)'}{\rho_{\mathbb H}}(z) \biggr |^2 \, dx \le Ce^{-R/2},
  \end{equation}
  where $L$ is a horizontal line segment which has hyperbolic diameter $R/2$.
\end{theorem}

The crucial feature of hyperbolic geometry that we use is that {\em a horocycle connecting two points is exponentially longer than the geodesic.} 
Thus, $L$ is extremely long: its length (as measured in the hyperbolic metric) is comparable to $e^{R/2}$.

For convenience, let us denote the horizontal lines in $\mathbb{C}$ by
 $$\ell_y = \{z : \im z = y\}, \qquad \text{and} \qquad \bar{\ell}_y = \{z : \im z = -y\}.$$
We need an elementary lemma, which is an exercise in Fubini's theorem:

\begin{lemma}
\label{small-intersection}
Suppose $\mu$ is a Beltrami coefficient on $\Hbar$, with $\|\mu\|_\infty \le 1$, such that the length of the intersection of 
any horizontal line $\bar{\ell}_y$ with $\supp \mu$ is bounded by $M$.
Then, for any $y > 0$, 
$$
 \int_{\ell_y} \, \biggl |\frac{(\mathcal S\mu)'}{\rho_{\mathbb H}}(z) \biggr | \, dx \le \frac{8}{\pi} \cdot M.
 $$
\end{lemma}

\begin{proof}[Proof of Theorem \ref{boxcart-sparse}.]
From the scale-invariance of the problem, we may assume that $L = \bigl [i, i+|L| \bigr ]$, where $|L| \asymp e^{R/2}$.
We divide the crescents $\{A_j\}_{j \in J}$ into two groups.
Group 1 consists of crescents that lie wholly above $\overline{L}$, that is,
$$A_j \subset \bigl [0, |L| \bigr ] \times [-1,0].$$
The remaining crescents form Group 2.
By assumption, there can be at most one special crescent that crosses $\overline{L}$. We denote it by $A_\times$ if it exists. It necessarily belongs to
the second group.

We write $\mu = \mu_1 + \mu_2$
 where $\supp \mu_1 \subseteq \mathcal G_1 = \bigcup_{j \in J_1} A_j$ and $\supp \mu_2 \subseteq \mathcal G_2 = \bigcup_{j \in J_2} A_j$.  
 Expanding the square and
 using Lemma \ref{qbounds-H}, we get
  \begin{equation}
\label{eq:sum-of-two-things}
 \int_{L} \, \biggl |\frac{2(\mathcal S\mu)'}{\rho_{\mathbb H}}(z) \biggr |^2 \, dx \le 
  \int_{L} \, \biggl |\frac{2(\mathcal S\mu_2)'}{\rho_{\mathbb H}}(z) \biggr |^2 \, dx
  +
   \frac{96}{\pi} \int_{L} \, \biggl |\frac{(\mathcal S\mu_1)'}{\rho_{\mathbb H}}(z) \biggr | \, dx.
     \end{equation}
To complete the proof, we need to show that both summands in (\ref{eq:sum-of-two-things}) are $\mathcal O(1)$.
For the first summand, it suffices to note that $\supp \mu_2$ is contained
in the union of the half-planes $\{w: \re z < c_1\}$, $\{w: \re z > e^{-R/2}-c_1\}$ and the special crescent $A_\times$ if it exists.
With this information, the estimate is settled by Lemma \ref{square-locality}. The reader may find it helpful to note that 
$A_\times$ is contained in a hyperbolic $\mathcal O(1)$ neighbourhood of two vertical lines.
For the second summand, we appeal to Lemma \ref{small-intersection}, where we use the bound
$|\mathcal G_1 \cap \bar{\ell}_y| = \mathcal O(1).$
As observed in \cite[Section 6]{ivrii-phd}, the hyperbolic length of the intersection
$A_j  \cap \ell_y$ is $\mathcal O(1)$, but as soon as soon as $\bar{\ell}_y$ intersects a crescent $A_j$, a
segment of hyperbolic length $\mathcal O(e^{R/2})$ must be disjoint from the other crescents.
\end{proof}